\documentclass[12pt,halfline,a4paper]{ouparticle}
\usepackage{amsmath}
\usepackage{amsfonts}
\usepackage{color}
\usepackage{tikz}
\usepackage{bussproofs}
\usepackage{scalerel}

\newtheorem{example}{Example}
\newtheorem{definition}{Definition}
\newtheorem{lemma}{Lemma}
\newtheorem{theorem}{Theorem}
\newtheorem{corollary}{Corollary}

\newtheorem{proof}{Proof}

\newcommand{\mvLKIE}{\textit{mv}\textbf{LKIE}}
\newcommand{\AxPa}{Ax$_{\rm PA}$}
\newcommand{\EPa}{$\mathcal{E}_{\rm PA}$}
\newcommand{\mvLKE}{\textit{mv}\LKE}
\newcommand{\mvLKS}{\textit{mv}\LKS}

\newcommand{\LK}{\textbf{LK}}
\newcommand{\LKE}{\textbf{LKE}}
\newcommand{\LKS}{\textbf{LKS}}
\newcommand{\PA}{\textbf{PA}}
\newcommand{\PRA}{\textbf{PRA}}

\newcommand*{\shifttext}[2]{%
  \settowidth{\@tempdima}{#2}%
  \makebox[\@tempdima]{\hspace*{#1}#2}%
}
\newcommand\fat[1]{\ThisStyle{\ooalign{%
  \kern.46pt$\SavedStyle#1$\cr\kern.33pt$\SavedStyle#1$\cr%
  \kern.2pt$\SavedStyle#1$\cr$\SavedStyle#1$}}}

\makeatletter
\newcommand*{\shifttexxt}[2]{%
  \settowidth{\@tempdima}{#2}%
  \makebox[\@tempdima]{\hspace*{#1}#2}%
}
\makeatother
\makeatletter
\newcommand\dotover{\leavevmode\cleaders\hb@xt@ .22em{\hss $\cdot$\hss}\hfill\kern\z@}
\newcommand{\dotfrac}[2]{
\ooalign{$\genfrac{}{}{0pt}{0}{#1}{#2}$\cr\dotover\cr}
}
\makeatother



\begin{document}

\title{Proof Schemata for Theories equivalent to PA: on the Benefit of Conservative Reflection Principles}

\author{
\name{David M. Cerna\footnote{This research is supported by the FWF project P28789-N32.}}
\address{Johannes Kepler University,\\ Research Institute for Symbolic Computation,\\ Linz, Austria}
\email{david.cerna@risc.jku.at} 
\and
\name{Anela Lolic\footnote{This research is supported by the FWF project I-2671-N35.}}
\address{Technical University of Vienna,\\ Institute of Discrete Mathematics and Geometry,\\ Vienna, Austria}
\email{anela@logic.at} 
}

\abstract{
Induction is typically formalized as a rule or axiom extension of the \textbf{LK}-calculus. While this extension of the sequent calculus is simple and elegant, proof transformation and analysis can be quite difficult. Theories with an induction rule, for example Peano arithmetic do not have a {\em Herbrand theorem}. In this work we extend an existing meta-theoretic formalism, so called {\em proof schemata}, a recursive formulation of induction particularly suited for proof analysis, to Peano arithmetic. This relationship provides a meaningful conservative reflection principle between \PA\ and an alternative proof formalism. Proof schemata have been shown to have a variant of {\em Herbrand's theorem} for classical logic which can be lifted to the subsystem of our new formalism equivalent to primitive recursive arithmetic. 
}


\maketitle

\section{Introduction}
Proof schemata serve as an alternative proof formulation through primitive recursive proof specification. Essentially, the local soundness of an individual {\em proof component} is sacrificed, but globally the collection of components is sound. This property is illustrated in~\cite{DCerna2017a} where a calculus integrating this global soundness is introduced. The seminal work concerning ``proof as schema''  was focused on proof analysis of F\"{u}rstenberg's proof of the infinitude of primes by Baaz et al.~\cite{MBaaz2008a} using a rudimentary schematic formalism and  \textbf{CERES}~\cite{MBaaz2000}. Schematic proof representation  excels at proof analysis and transformation without ``unrolling'' the formal proof, thus making it particularly suited for analysis of inductive reasoning. For example, Herbrand's theorem can be extended to an expressive fragment ($k$-induction) of proof schemata~\cite{CDunchev2014,ALeitsch2017}. We will discuss this extension of Herbrand's theorem and how to apply it to a fragment of the formalism introduced in this work, i.e. the formalism equivalent to {\em primitive recursive arithmetic}.  

While $k$-induction lacks in expressive power when compared to theories such as \PA~\cite{GTakeuti1975}, there is ample evidence showing that it is provability-wise quite expressive\footnote{See Gentzen~\cite{GGentzen1969}, the $\mathcal{S}\mathit{i}\mathbf{LK}$-calculus Cerna\ \& Lettmann~\cite{DCerna2017a}, and Curry's formalization of primitive recursive arithmetic~\cite{HCurry1941}.}. However, the formalism is quantifier free over the numeric sort and severely restricts the  proof structure, which is of  interest to proof analysis and transformation. For example, the infinitary pigeonhole principle (IPP)~\cite{DCerna2015,DCerna2016}\footnote{The infinitary pigeonhole principle can be stated as follows: a total function $f:\mathbb{N}\rightarrow N'$ s.t. $N'\subset \mathbb{N}$ and finite is not injective. } can be elegantly formalized as a  proof schema using $\Pi_2$-cuts, though an important proof skeleton resulting from proof analysis, using atomic cuts only,  is beyond the representational power of the current formal language. Most of the analysis had to be done outside the existing methodological framework~\cite{DCerna2015}. Analysis of a restricted version of IPP~\cite{DCerna2016} resulted in a proof skeleton expressible by the formal language presented in~\cite{CDunchev2014}. However, this restriction was designed to allow analysis in the existing framework by removing the complex combinatorics from which the full IPP statement is entailed (see~\cite{DCerna2015,DCerna2016}). Once again, the most interesting part of the proof had to be removed to fit the  constraints of the formal system.
 
Other existing alternative formalisms for arithmetic~\cite{JBrotherston2005,JBrotherston2010,RMcdowell1997} were developed without a concern for proof analysis and thus lack the construction of proof normal forms with {\em subformula-like properties}\footnote{A proof fulfilling the subformula property can be referred to as {\em analytic}. By subformula-like, we mean that the proof is non-analytic, but still allows the extraction of objects important for proof analysis which require some form of analyticity.} and analytic tools applicable to proofs in a compressed state, i.e. prior to unrolling their recursive specification. 

A recently developed proof analysis method~\cite{ALeitsch2017} based on the schematic theorem proving work pioneered by V. Aravantinos et al.~\cite{VAravantinos2011,VAravantinos2013} manages to preserve the analyticity, usually associated with the subformula property, in the presence of propositional cuts. While in some sense the method is cut-elimination complete in that it can provide the substitutions needed to construct a Herbrand sequent, it relies on a schematic superposition prover~\cite{VAravantinos2013} which cannot deal with the complexities of most interesting $k$-inductive arguments. Even though $k$-induction is quite expressive in theory, in practice we cannot expect the method of~\cite{ALeitsch2017} to perform proof analysis on arbitrarily complex arithmetic statements. Also, the complexity of the formalism introduced in~\cite{ALeitsch2017} makes the experimental approach taken in~\cite{DCerna2016} quite a challenging endeavor, an approach which has shown promise~\cite{DCerna2015}. Moreover, the formalism is based on a single numeric index, inductive basecases are not allowed and heavy restrictions are placed on nested inductions. Though this is not necessarily a weakness~\cite{DCerna2017a,GGentzen1969}, it does remove quantification over inductive arguments, the essential difference between \PA and Primitive Recursive Arithmetic (\PRA)~\cite{HCurry1941}.

In this work we provide a formalism which is provability-wise at least as expressive as \textbf{PA} without restricting the structure of the proof or the inductive argument. We do so by constructing proof schemata over a well-founded ordering and allowing multiple free parameters. Multi-parameter schemata provide quantifier introduction over numeric terms without the loss of properties essential for schematic proof analysis, and thus allowing us to formalize strong totality statement\footnote{As noted in~\cite{SHetzl2017}, a single inductively introduced quantifier suffices for formalizing \PA.}. An essential proof analytic tool of~\cite{CDunchev2014}, schematic characteristic clause set, can be generalized to our formalism and insofar support the experimental approach of~\cite{DCerna2016}. Though, a corresponding schematic resolution calculus (as defined in~\cite{CDunchev2014})  has yet to be developed. We plan to address this in future work by applying the methods and techniques developed here to schematic resolution refutations. We expect such investigations to provide further insight into automated deduction for recursively defined formulas and clause sets and provide a theoretical foundation supporting investigation into interactive theorem proving methods for schematic resolution. Furthermore, proof analysis carried out in our formalism will in all likelihood provide interesting and complex examples (as it has done so far) fueling the research areas mentioned above.

Furthermore, proof schemata as presented in this work provide a formalism allowing  induction over arbitrary function symbols, not just the successor. While we do not take advantage of this property in this paper it is a result of the recursive formalization presented which does not restrict the terms passed through {\em links} other then requiring them to be ordered. This implies that one can consider using the system for arbitrary induction definitions, which has not been considered in previous work. Another note concerning the practical uses of our formulation is that it produces a non-trivial conservative reflection principle. For uses of reflection principles with respect to arithmetic see Parikh's results~\cite{RParikh1973,MBaaz2008b} for the monadic version of {\PA}, that is \textbf{PA}$^{*}$.  We foresee application of similar methodology to other problems and see our approach as a method of formulating logical relationships between a classical formulation of a theory and its schematic counterpart.

\section{From Proof Schema to $P$-schema}
\label{FromPStoNS}

In this section we introduce  the previous formalization of the schematic language and proof schemata,  introduce our new formalism of $P$-schema, and provide a comparison of the two formalisms.    

We work in a setting with the sort $\omega$ representing the natural numbers, $o$ representing bool. The language consists of countable sets of variables and sorted $n$-ary function and predicate symbols. We associate with every $n$-ary function $f$ a tuple of sorts $(\tau_1 , \ldots , \tau_n , \tau)$ with the interpretation $f : \tau_1 \times \ldots \times \tau_n \to \tau$, analogously we do the same for  predicate symbols.

Terms are built from variables and function symbols in the usual way. We assume the constant function symbols $0: \omega$ and $s: \omega \to \omega$ (zero and successor) to be present. Formulas are built as usual from atoms using the logical connectives $\neg, \land, \lor, \to, \forall, \exists$.

\subsection{Proof Schemata and the Schematic Language}
\label{sec:basic}
We now introduce \emph{proof schemata} as presented in~\cite{DCerna2017a,DCerna2017b,CDunchev2014,ALeitsch2017}. In addition to the sort of individual term $\iota$ we also need a sort of {\em numeric terms} denoted by $\omega$. However, we limit ourselves to predicate symbols with at most a single numeric index which we assume to be the left most argument. In later sections we consider predicates with multiple numeric indices. Numeric terms are ground terms $\mathcal{G}_{\mathbb{N}}$ constructed from  the alphabet $\lbrace 0 , s(\cdot) \rbrace$ with the addition of a free parameter denoted by $n$. By $\mathcal{V}(k)$, where $k$ is a numeric term, we denote the set of parameters used in $k$. Note that we will use lower-case  Greek characters $\alpha, \beta, \gamma$ to represent ground numeric terms. 

We work with a {\em schematic first-order language} allowing the specification of an (infinite) set of first-order formulas by a finite term. Therefore we allow \emph{defined function symbols}, i.e. primitive recursively defined functions, in the language. Analogously, we use {\em defined predicate symbols} to build {\em formula schemata}, a generalization of formulas including defined predicate symbols defined inductively using the standard logical connectives from uninterpreted and defined predicate symbols. Defined symbols will be denoted by $\hat{\cdot}$, i.e. $\hat{P}$. We assume a set of convergent rewrite rules $\mathcal{E}$ (equational theory) for defined function and predicate symbols. The rules of $\mathcal{E}$ are of the form $\widehat{f}(\bar{t}) = E$, where $\bar{t}$ contains no defined symbols, and either $\widehat{f}$ is a function symbol and $E$ is a term or $\widehat{f}$ is a predicate symbol and $E$ is a formula schema. The rules can be applied in both directions, i.e. the $\mathcal{E}$-rule is reversible. Such symbols can be defined  for both the $\omega$ and $\iota$ sort. 

\begin{example}
Iterated version of  $\lor$ and $\land$ operators ( the defined predicates are abbreviated 
as $\bigvee$ and $\bigwedge$ ) can be defined using the following equational theory: 
\begin{align*}
\bigvee_{i=0}^0P(i)=& P(0) & & \bigwedge_{i=0}^0P(i)=  P(0)\\
\bigvee_{i=0}^{s(y)}P(i)=&  \bigvee_{i=0}^{y}P(i)\lor P(s(y))& &
\bigwedge_{i=0}^{s(y)}P(i)= \bigwedge_{i=0}^{y}P(i)\land P(s(y)).
\end{align*}

We can also iterate function symbols of the $\iota$ sort as follows: 
\begin{align*}
It(s(n),a)= f(It(n,a)) & & It(0)= a
\end{align*}
where $f$ is a 1-ary function and $a$ a constant of the $\iota$.

\end{example}

A {\em schematic sequent} is a pair of multisets of formula sche-mata $\Delta$, $\Pi$ denoted by $\Delta \vdash \Pi$. We will denote multisets of formula schemata by upper-case Greek letters. We reuse this construction for $P$-schema. 

\begin{definition}[\LKE]
Let $\mathcal{E}$ be an equational theory. \LKE\ is an extension of \LK\ by the
$\mathcal{E}$ inference rule
\begin{small}\begin{prooftree}
\AxiomC{$S(t)$}
\RightLabel{$\mathcal{E}$}
\UnaryInfC{$S(t')$}
\end{prooftree}\end{small}
where the term or formula schema $t$ in the sequent $S$ is replaced by a term or formula schema $t'$ for 
$\mathcal{E}\models t= t'$.
\end{definition}

Let $S(\overline{x})$ be a sequent and $\overline{x}$ a vector of free variables, then $S(\overline{t})$ denotes $S(\overline{x})$ where $\overline{x}$ is replaced by $\overline{t}$, where $\overline{t}$ is a vector of terms of appropriate type. We denote by the following construction $${ \dotfrac{(\varphi,\overline{a})}{S(\overline{a})}}$$, where $S(\overline{x})$ is a schematic sequent and $\varphi$ a {\em proof symbol} from the countably infinite set of proof symbols $\mathcal{B}$, a so called {\em proof link}. Proof links are to be interpreted as $0$-ary inference rules acting as place holds for proofs. The sequent calculus {\LKS} consists of the rules of {\LKE} where the leaves of the proof tree can be axioms or  proof links.

 Proof schemata are a sequence of \emph{proof schema components} with a restriction on the usage of proof links. Note that proofs may have free parameters, that is variables of the numeric sort. In this section, we limit the number of parameters to one and show how this restriction can be lifted. Let $\nu$ be an {\LKS}-proof containing a free parameter $n$. By $\nu(k)$ we denote the {\LKS}-proof derived from $\nu$ by replacing the free parameter with a numeric expression $k$ such that. whenever an {\LKS}-proof $\nu$ contains a free parameter we write $\nu(n)$ or $\nu(k)$ where $k$ is a numeric expression, otherwise we write $\nu$. 

\begin{definition}[Proof schema component]
Let $\psi$ be a proof symbol and $k$ an expression with free variable $n$. A \emph{proof schema
component} $\mathbf{C}$  is a triple $(\psi,\pi ,\nu(k))$  where $\pi$ is an \LKS-proof without proof links and $\nu(k)$ is an \LKS-proof containing proof links and $k$ is a numeric expression such that $\mathcal{V}(k)\subseteq \left\lbrace n \right\rbrace$. The end-sequents of the proofs are $S(0,\overline{x})$ and $S(k,\overline{x})$, respectively. Given a proof schema component $\mathbf{C}=(\psi,\pi ,\nu(k))$ we define $\mathbf{C}.1 = \psi$, $\mathbf{C}.2 = \pi$, and  $\mathbf{C}.3 = \nu(k)$.
\end{definition}

\begin{definition}
Let $\mathbf{C}$ and $\mathbf{D}$ be proof schema components such that $\mathbf{C}.1$ is distinct from $\mathbf{D}.1$. We say $\mathbf{C}\succ\mathbf{D}$ if there are no links in $\mathbf{D}.2$ and $\mathbf{D}.3$ to $\mathbf{C}.1$ and all links in $\mathbf{C}.2$ and $\mathbf{C}.3$ to $\mathbf{D}.1$ are of the following form: 
\begin{center}
\begin{minipage}{.20\textwidth}
\begin{small}
\begin{prooftree}
\AxiomC{$(\mathbf{C}.1, k',\overline{r})$}
\dashedLine
\UnaryInfC{$S(k',\overline{r})$}
\end{prooftree}
\end{small}
\end{minipage}
\begin{minipage}{.24\textwidth}
\begin{small}
\begin{prooftree}
\AxiomC{$(\mathbf{D}.1, t,\overline{r})$}
\dashedLine
\UnaryInfC{$S'(t,\overline{r})$}
\end{prooftree}
\end{small}
\end{minipage}
\end{center}
for $t$ s.t. $\mathcal{V}(t)\subseteq \left\lbrace n \right\rbrace$, $k'$ is a sub-term of $k$, and $\overline{r}$ is a vector of terms of the appropriate sort. $S(x)$ 
and $S'(x)$ are the end sequents of components $\mathbf{C}_1$ and 
$\mathbf{C}_2$, respectively. Let $\Psi$ be a set of proof schema components. We say $\mathbf{C}\succ \mathbf{D}$ if $\mathbf{C}\succ_{\Psi}\mathbf{D}$ and $\mathbf{C}\succ_{\Psi} \mathbf{D}$ holds for all proof schema components $\mathbf{E}$ of $\Psi$ with $\mathbf{D}\succ_{\Psi}\mathbf{E}$.
\end{definition}

\begin{definition}[Proof schema \cite{CDunchev2014}]
Let $\mathbf{C}_1,\cdots, \mathbf{C}_m$ be a sequence proof schema components s.t the
$\mathbf{C}_i.1$ are distinct. Let the  end sequents of $\mathbf{C}_1$ be $S(0,\overline{x})$ and $S(k,\overline{x})$. We define  $\Psi 
= \left\langle \mathbf{C}_1, \cdots, \mathbf{C}_m \right\rangle $ as a 
{\em proof schema} if $\mathbf{C}_1\succ_{\Psi}\ldots\succ_{\Psi}\mathbf{C}_m$. We call $S(k,\overline{x})$ the end sequent 
of $\Psi$. 
\end{definition}

\begin{example}[Proof schema]
  Let us define a proof schema $\left\langle (\varphi, \pi, \nu(k)) \right\rangle $ with end se- quent (schema)
$$P(0), \bigwedge_{i=0}^n (P(i)\to P(i+1))\vdash P(n+1).$$
using the equational theory:
\[\mathcal{E} = \left\lbrace  \begin{array}{l}\bigwedge_{i=0}^0 P(i)\to P(i+1)= P(0)\rightarrow  P(1)  ; \\\\ \bigwedge_{i=0}^{n+1} P(i)\to P(i+1)= \\ P(n+1)\rightarrow  P(n+2) \wedge \bigwedge_{i=0}^n P(i)\to P(i+1)\end{array}\right\rbrace. \]
$\pi$ is as follows:
\begin{small}
\begin{prooftree}
  \AxiomC{$P(0)\vdash P(0)$}
  \AxiomC{$P(1)\vdash P(1)$}
  \RightLabel{$\to\colon l$}
  \BinaryInfC{$P(0), P(0)\to P(1) \vdash P(1)$}\RightLabel{$\mathcal{E}$}
  \UnaryInfC{$P(0), \bigwedge_{i=0}^0 P(i)\to P(i+1) \vdash P(1)$}
\end{prooftree}
\end{small}
$\nu(k)$ is as follows where $k=n+1$:

\begin{prooftree}
  \AxiomC{ $(\nu_1 (k))$}
  \noLine
  \UnaryInfC{$S(\nu_1 (k))$}
  \AxiomC{ $(\nu_2 (k))$}
  \noLine
  \UnaryInfC{$S(\nu_2 (k))$}
\RightLabel{$cut$}
  \BinaryInfC{$P(0), \bigwedge_{i=0}^n (P(i)\to P(i+1)), P(n+1)\to P(n+2)\vdash {P(n+2)} $}
  \UnaryInfC{$\ldots \wedge\colon l$ and $c\colon l\ldots$}
  \UnaryInfC{$P(0), \bigwedge_{i=0}^{n+1} (P(i)\to P(i+1))\vdash {P(n+2)} $}
\end{prooftree}
where 
$$S(\nu_1 (k)) \equiv P(0), \bigwedge_{i=0}^n (P(i)\to P(i+1))\vdash {P(n+1)},$$
$$S(\nu_2 (k)) \equiv P(n+1), P(n+1) \to P(n+2) \vdash P(n+2),$$
$\nu_1 (k)$ is
\begin{small}
\begin{prooftree}
	\AxiomC{ $\dotfrac{\varphi(n)}{ P(0), \bigwedge_{i=0}^n (P(i)\to P(i+1))\vdash {P(n+1)}}{\color{white}\Bigg\downarrow}$}
\end{prooftree}
\end{small}
and $\nu_2 (k)$ is
\begin{small}
\begin{prooftree}
	\AxiomC{$P(n+1) \vdash P(n+1)$}
	\AxiomC{$P(n+2) \vdash P(n+2)$}
	\BinaryInfC{ $P(n+1), P(n+1) \to P(n+2) \vdash P(n+2)$}
\end{prooftree}
\end{small}

\end{example}

\subsection{The $\mathbf{P}$-schema Formalism}
We extend the formalism defined in the previous section to construct {\em $\mathbf{P}$-schemata}. Other than the ground numeric terms we also have three types of parameters: {\em active} parameters $\mathcal{N}_{a}$, {\em passive} parameters $\mathcal{N}_{p}$, and {\em internal} parameters $\mathcal{N}_{i}$\footnote{A related terminology can be found in~\cite{HSimmons1988} which discusses a construction similar to ours.}. Intuitively, active parameters are the parameters associated with recursive construction, passive parameters are associated with a prior recursive construction and, internal parameters are used for auxiliary arguments to a recursive construction. Note that no changes are made to the $\iota$ sort concerning this generalization.

 We denote the active parameters by lower-case Latin characters $n,m,k$, passive parameters by lower-case bold Greek characters $\fat{\alpha}, \fat{\beta}, \fat{\gamma}$, and internal parameters by lower-case bold Latin characters  $\mathbf{n},\mathbf{m},\mathbf{k}$. We define the set of numeric terms containing these parameters as   $\mathcal{A}_{\mathbb{N}}$, $\mathcal{P}_{\mathbb{N}}$, $\mathcal{I}_{\mathbb{N}}$, respectively. In general, the set of schematic terms will be denoted by $\mathcal{S}_{\mathbb{N}} = \mathcal{A}_{\mathbb{N}}\cup \mathcal{P}_{\mathbb{N}}\cup \mathcal{G}_{\mathbb{N}}\cup \mathcal{I}_{\mathbb{N}}$. To simplify reading, we will denote the successor of an active parameter $n$ by $n'$ rather than $s(n)$. 

Concerning quantification, one ought to consider passive parameters as {\em eigenvariables}. Both active and internal parameters as well as terms which contain them cannot be quantified. These parameters play a computational role and thus quantification of these symbols can lead to unsound derivation.

We will be mainly concerned with schematic sequents $S$ referred to as $(n,\mathcal{I})$-sequent, where $n$ is an active parameter and $\mathcal{I}$ is a set of internal parameters, where the free active and internal parameters of $S$ are $n$ and $\mathcal{I}$, respectively. A sequent without active parameters will be referred to as an $\mathcal{I}$-sequent. We do not consider sequents with more than one active parameter. Our primitive recursive formalization requires pairing the end sequent of a base case proof (an $\mathcal{I}$-sequent) with the end sequent of a step case proof (an $(n,\mathcal{I})$-sequent); we call these pairs of sequents {\em inductive pairs}. For example, given an $(n,\mathcal{I})$-sequent $S$ an  inductive pair based on $S$ would be ($S$, $S\lbrace n\leftarrow \alpha \rbrace$) where $\alpha \in \mathcal{G}_{\mathbb{N}}$.

We consider an extension of the \LKE-calculus~\cite{CDunchev2014} over the extended term language defined above, which we refer to as the  {\em multivariate} \LKE-calculus (\mvLKE-calculus). By multivariate we are referring to the three types of variable symbols. Before moving on to the definition of $\mathbf{P}$-schemata we further distinguish the three types of parameters. By
active parameter we are referring to the parameter over which the induction is performed. Every \mvLKE-proof has at most one distinct active parameter. Passive parameters are used to mark already occurring inductions, or as mentioned before, act as  eigenvariables. There is no limit to the number of passive parameters  in a given proof. Internal parameters are used to pass information through the ``links'' which we will discuss shortly. Thus, an \mvLKE-derivation can end with a schematic sequent containing internal parameters, but an \mvLKE-proof ends with a schematic sequent without internal or active parameters. The end-sequent of an \mvLKE-proof (\mvLKE-derivation) $\varphi$ will be denoted as $es(\varphi)$ and the set $\mathcal{V}_{x}(S)$ for $x\in \lbrace a, p, i\rbrace$ will denote the active, passive and internal parameters occurring in the sequent $S$, respectively. Notice that the calculus introduced so far cannot construct \mvLKE-proofs unless the entire derivation is active and internal parameter free. To deal with this issue we introduce the concept of schematic proof (derivation) for the \mvLKE-calculus. 

As in Section \ref{sec:basic}, we define a schematic proof (derivation) as a finite set of {\em components} which can ``link'' together using {\em links} defined analogously to proof links. Note any type of parameter is allowed in the arguments of the links.    
Furthermore, we assume a countably infinite set $\mathcal{B}$\footnote{Think of the German word $\mathcal{B}$eweis meaning proof.} of \emph{proof symbols} denoted by $\varphi ,\psi ,\varphi_i,\psi_j$. 

\begin{definition}[\mvLKS]
The \mvLKS-calculus is an extension of \mvLKE, where links may appear at the leaves of a proof (derivation). 
\end{definition}

Note that \mvLKS-proofs do not have to be valid by construction and require an external soundness condition. This soundness condition is provided by the $\mathbf{P}$-schema construction.  Also, we will refer to an \mvLKS-derivation as {\em inactive} if it does not contain an active parameter and $\lbrace n \rbrace$-active if it contains only the active parameter $n$.

\begin{example} \label{ex.running1} Consider the following  $\mathcal{E}$ theory 
\[\mathcal{E} = \left\lbrace \widehat{a}(s(n),\beta) =s(\widehat{a}(n,\beta)) ; \ \widehat{a}(0,\beta) = \beta \right\rbrace. \]

Where $\hat{a}(\cdot, \cdot)$ represents addition.
Let $\pi =$
\begin{small}
\begin{prooftree}
\AxiomC{$  \vdash 0 = 0$}
\RightLabel{$\mathcal{E}$}
\UnaryInfC{$ \vdash \hat{a}(0,0) =0$}
\RightLabel{$\mathcal{E}$}
\UnaryInfC{$ \vdash \hat{a}(0,0) =\hat{a}(0,0)$}
\end{prooftree}
\end{small}
and let $\nu =$
\begin{small}
\begin{prooftree} 
\AxiomC{$\chi(n)$}
\dottedLine
\UnaryInfC{$\vdash \hat{a}(n,0) =\hat{a}(0,n)$}
\AxiomC{$S_1(\nu_1 )$}
\UnaryInfC{$\vdots$}
\RightLabel{$cut$}
\BinaryInfC{$\vdash  \hat{a}(n',0) =\hat{a}(0,n')$}
\end{prooftree}
\end{small}
where $S_1 (\nu_1 ) \equiv \hat{a}(n,0) =\hat{a}(0,n) \vdash s(\hat{a}(n,0)) =s(\hat{a}(0,n))$. $\pi$ is an \mvLKS-proof and $\nu$ is an  \mvLKS-derivation. Also,  $\nu$ contains a link to the proof symbol $\chi$ and is $\lbrace n \rbrace$-active. Note that the end-sequent of $\nu$, $es(\nu)$ is $es(\chi(n))\{n \leftarrow n'\}$. Moreover, $es(\nu)$ is an $(n,\emptyset)$-sequent, i.e. internal parameter free. Its inductive pair also contains the sequent $es(\nu)\{n \leftarrow 0\} = \ \vdash \hat{a}(0,0) = \hat{a}(0,0)$, the end-sequent of $\pi$ (an $\emptyset$-sequent). The triple $(\chi, \pi_1 , \nu_1)$ is referred to as an $(n,\emptyset)$-component. 
\end{example}

\begin{definition}[$(n,\mathcal{I})$-component]
Let $\psi\in \mathcal{B}$, $n\in \mathcal{N}_a$ and $\mathcal{I}\subset \mathcal{N}_i$. An \emph{$(n,\mathcal{I})$-component} $\mathbf{C}$  is a triple $(\psi,\pi ,\nu)$  where $\pi$ is an inactive \mvLKS-derivation ending with $S\{n \leftarrow \alpha\}$ and $\nu$ is an $\lbrace n\rbrace$-active \mvLKS-derivation ending in an $(n,\mathcal{I})$-sequent $S$, the inductive pair of $S\{n \leftarrow 0\}$. Given a 
component $\mathbf{C}=(\psi,\pi ,\nu)$ we define $\mathbf{C}.1 = \psi$, 
$\mathbf{C}.2 = \pi$, 
and  $\mathbf{C}.3 = \nu$. We refer to $es(\mathbf{C}) = S$ as the {\em end sequent} of the  component $\mathbf{C}$.
\end{definition}
When the extra terms are not necessary for understanding, we will refer to an $(n,\mathcal{I})$-component as a component.

A schematic proof is defined over finitely many components, that are connected by links. This means, that a component may contain links to other components. So far, there is no restriction on the usage of links, which must obey some conditions in order to preserve soundness. In fact, whenever a component $\mathbf{C}$ contains a link to another component $\mathbf{D}$, it has to be ensured that the passive parameters occurring in $es(\mathbf{D})$ occur in the sequent associated with the link in $\mathbf{C}$ as well. This condition is formalized in the following definition. 

\begin{definition}[Linkability]
Two components  $\mathbf{C}$ and $\mathbf{D}$ are said to be {\em $(\mathbf{C},\mathbf{D}$)-linkable} if for each sequent $S$ in $\mathbf{C}$ that is associated with a proof link to $\mathbf{D}$ it holds that $\mathcal{V}_{p}(es(\mathbf{D})) \subseteq \mathcal{V}_{p}(S)$. We say they are {\em strictly $(\mathbf{C},\mathbf{D}$)-linkable} if it holds that $\mathcal{V}_{p}(es(\mathbf{D})) \subseteq \mathcal{V}_{p}(es(\mathbf{C}))$. 
\end{definition}

The restriction on linkable components is used to define an ordering on the components occurring in a schematic proof.
\begin{definition}[Linkability ordering]
\label{def:ordering}
Let $\mathbf{C}_1$  and $\mathbf{C}_2$ be distinct components such that they are $(\mathbf{C}_1,\mathbf{C}_2)$-linkable. Then we say that $\mathbf{C}_1 \prec \mathbf{C}_2$. If $\mathbf{C}_1$ and $\mathbf{C}_2$ are strictly $(\mathbf{C}_1,\mathbf{C}_2)$-linkable, we say that $\mathbf{C}_1 \prec_s \mathbf{C}_2$.
\end{definition}

When constructing recursive proofs as introduced in this paper avoiding mutual recursion is essential being that unfolding of the proof can produce infinite cycles, i.e. infinitely large proof trees. The restriction on the number of active parameters per sequent deals with mutual recursion and the linkability ordering provides a method to deal with infinite cycles\footnote{We do not consider proof by infinite descent~\cite{JBrotherston2010}.}. However, the linkability ordering alone does not prevent infinite cycles; it is defined over the set of all components and allows the definition of mutually linkable components, i.e. $(\mathbf{C}_1,\mathbf{C}_2)$-linkable and $(\mathbf{C}_2,\mathbf{C}_1)$-linkable components $\mathbf{C}_1,\mathbf{C}_2$. To avoid this problem we define $\mathbf{P}$-schemata over a subordering of the linkability ordering which is well founded. In this work we restrict this well founded suborder to a finite set of objects, as this suffices for the presented results. However, a generalization to more complex well orderings is possible and is worth future investigation.  

\begin{definition}[$\mathbf{P}$-schema]
Let $\mathbf{P} \subset \mathcal{N}_p$, $\mathbf{C}_1$ an $(n,\mathcal{I})$-component and $\mathbf{C}_2,\cdots, \mathbf{C}_\alpha$ components such that for all $1\leq i\leq \alpha$, 
$\mathbf{C}_i.1$ are distinct and $\mathcal{V}_{p}(\mathbf{C}_i)\subseteq \mathbf{P}$. We define  $\Psi = \left\langle \mathbf{C}_1, \cdots, \mathbf{C}_\alpha \right\rangle $ as a {\em $\mathbf{P}$-schema} ({\em strict $\mathbf{P}$-schema}) over a well founded suborder  $\prec^*\subset \prec$ ( $\prec_{s}^*\subset \prec_{s}$) of $\left\lbrace \mathbf{C}_1,\ldots,\mathbf{C}_\alpha\right\rbrace$ with $\mathbf{C}_1$ as  least element. We define $|\Psi|= \alpha$, $\Psi.i = \mathbf{C}_i$ for $1\leq i \leq \alpha$, and $es(\Psi) = es(\mathbf{C}_1)$. 
\end{definition}
\begin{definition}[sub $\mathbf{P}$-schema]
Let \ $\Psi = \left\langle \mathbf{C}_1, \cdots, \mathbf{C}_\alpha \right\rangle $ be a  $\mathbf{P}$-schema and $\Psi' = \left\langle \mathbf{C}_1', \cdots, \mathbf{C}_\beta' \right\rangle $ be a  $\mathbf{P}$-schema such that $\Psi' \subseteq \Psi$. We refer to $\Psi'$ as a {\em sub $\mathbf{P}$-schema} of $\Psi$ if the following hold: 
\begin{itemize}
\item[(1)] $\bigcup_{C\in \Psi} \mathcal{V}_{a}(es(C)) \cap \bigcup_{C\in \Psi'} \mathcal{V}_{a}(es(C)) = \emptyset$
\item[(2)] If some component $C\in \Psi$ links to a component $C'\in \Psi'$, substituting  $n\in \mathcal{V}_{a}(es(C'))$ by a term $t$, then $\mathcal{V}_{a}(t) \cap \bigcup_{C\in \Psi\setminus \Psi'} \mathcal{V}_{a}(es(C)) = \emptyset$ and $\mathcal{V}_{p}(t) \cap \mathcal{V}_{p}(es(\Psi))$ $=\emptyset$, and $\mathcal{V}_{i}(t)=\emptyset$.  
\end{itemize}
\end{definition}
 
 We refer to a sub $\mathbf{P}$-schema as {\em computational} if condition (2) is strengthen to ``. . . substituting  $n\in \mathcal{V}_{a}(es(C'))$ by a term $t\not \in \mathcal{G}_{N}$ . . . ''

\begin{example}
\label{example.1}
Using basic equational reasoning we can formalize  associativity of addition as an $\lbrace\fat{\alpha},\fat{\beta},\fat{\gamma}\rbrace$-schema $\Phi = \left\langle (\varphi,\pi,\nu ) \right\rangle$ over the following  $\mathcal{E}$ theory 
\[\mathcal{E} = \left\lbrace \widehat{a}(s(n),\beta) =s(\widehat{a}(n,\beta)) ; \ \widehat{a}(0,\beta) = \beta \right\rbrace. \]

$\pi =$
\begin{small}
\begin{prooftree}
\AxiomC{$\begin{array}{c} \vdash  \hat{a}(\mathbf{k},\fat{\gamma}) =\hat{a}(\mathbf{k},\fat{\gamma}) \end{array}$}
\RightLabel{$\mathcal{E}$}
\UnaryInfC{$\begin{array}{c} \vdash  \hat{a}(0,\hat{a}(\mathbf{k},\fat{\gamma})) =\hat{a}(\mathbf{k},\fat{\gamma})  \end{array}$}
\RightLabel{$\mathcal{E}$}
\UnaryInfC{$\begin{array}{c} \vdash  \hat{a}(0,\hat{a}(\mathbf{k},\fat{\gamma})) =\hat{a}(\hat{a}(0,\mathbf{k}),\fat{\gamma})  \end{array}$}
\end{prooftree}
\end{small}
$\nu=$
\begin{small}
\begin{prooftree}
\AxiomC{$\begin{array}{c} \varphi(n,\mathbf{k},\fat{\gamma}) \end{array}$}
\dottedLine
\UnaryInfC{$\begin{array}{c} \vdash  \hat{a}(n,\hat{a}(\mathbf{k},\fat{\gamma})) =\hat{a}(\hat{a}(n,\mathbf{k}),\fat{\gamma})  \end{array}$}
\AxiomC{$(\nu_1)$}
\noLine
\UnaryInfC{$S(\nu_1)$}
\RightLabel{$cut$}
\BinaryInfC{$\begin{array}{c} \vdash   \hat{a}(n',\hat{a}(\mathbf{k},\fat{\gamma})) =\hat{a}(\hat{a}(n',\mathbf{k}),\fat{\gamma})\\ \end{array}$}
\end{prooftree}
\end{small}
where $S(\nu_1 ) \equiv \hat{a}(n,\hat{a}(\mathbf{k},\fat{\gamma})) = \hat{a}(\hat{a}(n,\mathbf{k}),\fat{\gamma}) \vdash \hat{a}(n',\hat{a}(\mathbf{k},\fat{\gamma})) = \hat{a}(\hat{a}(n',\mathbf{k}),\fat{\gamma})$ and $\nu_1$ is
\begin{small}
\begin{prooftree}
\AxiomC{$\begin{array}{c}  \hat{a}(n,\hat{a}(\mathbf{k},\fat{\gamma})) =\hat{a}(\hat{a}(n,\mathbf{k}),\fat{\gamma})\vdash \\  s(\hat{a}(n,\hat{a}(\mathbf{k},\fat{\gamma}))) =s(\hat{a}(\hat{a}(n,\mathbf{k}),\fat{\gamma})) \end{array}$}
\RightLabel{$\mathcal{E}$}
\UnaryInfC{$\begin{array}{c}  \hat{a}(n,\hat{a}(\mathbf{k},\fat{\gamma})) =\hat{a}(\hat{a}(n,\mathbf{k}),\fat{\gamma})\vdash \\  \hat{a}(n',\hat{a}(\mathbf{k},\fat{\gamma})) =s(\hat{a}(\hat{a}(n,\mathbf{k}),\fat{\gamma})) \end{array}$}
\RightLabel{$\mathcal{E}$}
\UnaryInfC{$\begin{array}{c}  \hat{a}(n,\hat{a}(\mathbf{k},\fat{\gamma})) =\hat{a}(\hat{a}(n,\mathbf{k}),\fat{\gamma})\vdash  \\ \hat{a}(n',\hat{a}(\mathbf{k},\fat{\gamma})) =\hat{a}(s(\hat{a}(n,\mathbf{k})),\fat{\gamma}) \end{array}$}
\RightLabel{$\mathcal{E}$}
\UnaryInfC{$\begin{array}{c}  \hat{a}(n,\hat{a}(\mathbf{k},\fat{\gamma})) =\hat{a}(\hat{a}(n,\mathbf{k}),\fat{\gamma})\vdash \\  \hat{a}(n',\hat{a}(\mathbf{k},\fat{\gamma})) =\hat{a}(\hat{a}(n',\mathbf{k}),\fat{\gamma}) \end{array}$}
\end{prooftree}
\end{small}

Notice that $\nu$ is an \mvLKS-derivation not an \mvLKS-proof being that the end sequent of $\nu$ is $\lbrace n \rbrace$-active.  We can extend $\Phi$ to $\Phi^* = \left\langle (\chi,\lambda,\mu ) , (\varphi,\pi,\nu ) \right\rangle$ where 

\begin{minipage}{.2\textwidth}
$\lambda =$
\end{minipage}
\begin{minipage}{.8\textwidth}
\begin{prooftree}
\AxiomC{$\begin{array}{c} \varphi(0,\fat{\beta},\fat{\gamma}) \end{array}$}
\dottedLine
\UnaryInfC{$\begin{array}{c} \vdash   \hat{a}(0,\hat{a}(\fat{\beta},\fat{\gamma})) =\hat{a}(\hat{a}(0,\fat{\beta}),\fat{\gamma}) \end{array}$}
\end{prooftree}
\end{minipage}

\begin{minipage}{.2\textwidth}
$\mu =$
\end{minipage}
\begin{minipage}{.8\textwidth}
\begin{prooftree}
\AxiomC{$\begin{array}{c} \varphi(\fat{\alpha},\fat{\beta},\fat{\gamma}) \end{array}$}
\dottedLine
\UnaryInfC{$\begin{array}{c} \vdash   \hat{a}(\fat{\alpha},\hat{a}(\fat{\beta},\fat{\gamma})) =\hat{a}(\hat{a}(\fat{\alpha},\fat{\beta}),\fat{\gamma}) \end{array}$}
\end{prooftree}
\end{minipage}

The resulting schema $\Phi^*$ ends with an \mvLKS-proof and thus constructs an infinite sequence of \mvLKS-proof.  
\end{example}

 Note that this formalization is a generalization of the formalization described in~\cite{ALeitsch2017}. For example, if we were only to use $(n, \mathcal{I} )$-components and construct only strict  $\lbrace\fat{\alpha}\rbrace$-schema, the resulting proof structure would be equivalent to first-order proof schemata with a restricted $\iota$ sort, i.e. restricted to numerals. Even though the $\lbrace\fat{\alpha},\fat{\beta},\fat{\gamma}\rbrace$-schema $\Phi$ provided in Example~\ref{example.1} has an $\lbrace n \rbrace$-active end sequent with a free internal parameter, these are nothing more than the free parameter and a free variable of the $\iota$ sort as discussed in~\cite{ALeitsch2017} and thus, this example is easily expressible within that formalization. We can extend this example to a proof of commutativity which is beyond the expressive power of previous formalizations. Note that to prove commutativity we need to allow axioms of the equational extension of the \LK-calculus in our \LKS-proofs. In Chapter 1.7 of \cite{GTakeuti1975} such an extension is referred to as the \LK$_{e}$-calculus calculus and thus we refer to our extension as the \LKS$_{e}$-calculus. The equational axiom schemes are the following: 
 \begin{align*}
 \ \vdash\ & s = s \\ 
  s=t\ \vdash\ & s(s) =  s(t)\\
 s_{1}=t_{1},\cdots , s_{k}=t_{k} \ \vdash\ & \hat{f}(s_{1},\cdots, s_{k},\overline{r}) =  \hat{f}(t_{1},\cdots, t_{k},\overline{r})\\
  s_{1}=t_{1},\cdots, s_{k}=t_{k},\ R(s_{1},\cdots, s_{k},\overline{r}) \ \vdash\ &   R(t_{1},\cdots, t_{k},\overline{r})\\
 \end{align*}
 where $s,s_{1},\cdots, s_k,t,t_1,\cdots t_k$ are numeric terms, $\overline{r}$ is a vector of terms of the individual sort,  $\hat{f}$ is a defined function symbols, $R$ is either a defined predicate symbol or predicate symbol, and $s(\cdot)$ is the successor of the numeric sort. 
\begin{example}
\label{ex:two}
We use the same $\mathcal{E}$ theory as presented in Example~\ref{example.1} and extend the 
$\lbrace \fat{\alpha}, \fat{\beta}, \fat{\gamma}\rbrace$-schema of Example~\ref{example.1} to the $\lbrace\fat{\alpha}, \fat{\beta}\rbrace$-schema
$$\Phi' = \left\langle (\chi,\pi_1,\nu_1 ) ,(\psi,\pi_2,\nu_2 ), (\xi,\pi_3,\nu_3 ) (\varphi,\pi,\nu ) \right\rangle$$
using the following equational axioms:
\begin{align*}
E_1 \equiv & \hat{a}(\fat{\alpha},1) = \hat{a}(1,\fat{\alpha}) \vdash \hat{a}(\hat{a}(\fat{\alpha},1),n) = \hat{a}(\hat{a}(1,\fat{\alpha}),n)\\
E_2 \equiv & \hat{a}(\hat{a}(1,\fat{\alpha}),n) = \hat{a}(n',\fat{\alpha}), \hat{a}(\hat{a}(\fat{\alpha},1),n) = \hat{a}(\hat{a}(1,\fat{\alpha}), n) \vdash 
            \hat{a}(\hat{a}(\fat{\alpha},1),n) = \hat{a}(n',\fat{\alpha})\\
E_3 \equiv & \hat{a}(\fat{\alpha},n) =\hat{a}(n,\fat{\alpha}) \vdash s(\hat{a}(\fat{\alpha},n)) =s(\hat{a}(n,\fat{\alpha}))\\
E_4 \equiv & \hat{a}(\fat{\alpha}, \hat{a}(1,n)) = \hat{a}(\hat{a}(\fat{\alpha},1),n) , \hat{a}(\hat{a}(\fat{\alpha},1),n) = \hat{a}(n',\fat{\alpha})\vdash 
            \hat{a}(\fat{\alpha}, \hat{a}(1,n)) = \hat{a}(n',\fat{\alpha})
\end{align*}
Note the these equational axioms are either instances of the axiom schemes of $\LKS_{e}$ or derivable from them. We define 
$\pi_1$ and $\nu_1$  as in Example \ref{ex.running1}. $\pi_2 =$
\begin{small}
\begin{prooftree}
\AxiomC{$\vdash s(0)= s(0)$}
 \UnaryInfC{$    \vdash  \hat{a}(0,s(0))= s(0)$}
\RightLabel{$\mathcal{E}$}
\UnaryInfC{$    \vdash  \hat{a}(0,s(0))=  s(\hat{a}(0,0))$}
\RightLabel{$\mathcal{E}$}
\UnaryInfC{$    \vdash  \hat{a}(0,s(0))= \hat{a}(s(0),0)$}
\end{prooftree}
\end{small}
$\nu_2=$
\begin{small}
\begin{prooftree}
\AxiomC{$\psi(n)$}
\dottedLine
\UnaryInfC{$\vdash  \hat{a}(n,1)= \hat{a}(1,n)$}
\AxiomC{$\hat{a}(n,1)= \hat{a}(1,n) \vdash s(\hat{a}(n,1))= s(\hat{a}(1,n))$}
\RightLabel{$cut$}
\BinaryInfC{$\vdots$}
\UnaryInfC{$\vdash \hat{a}(n',1)= \hat{a}(1,n')$}
\end{prooftree}
\end{small}
 $\pi_3 = $
\begin{small}
\begin{prooftree}
\AxiomC{$ \chi(\fat{\alpha})$}
\dottedLine
\UnaryInfC{$ \vdash \hat{a}(\fat{\alpha},0) =\hat{a}(0,\fat{\alpha})$}
\end{prooftree}
\end{small}
$\nu_3=$
\begin{prooftree}
\AxiomC{$\xi(n,\fat{\alpha})$}
\dottedLine
\UnaryInfC{$    \vdash  \hat{a}(\fat{\alpha},n) =\hat{a}(n,\fat{\alpha}) $}
\UnaryInfC{$ \vdots $}
\AxiomC{$\psi(\fat{\alpha})$}
\dottedLine
\UnaryInfC{$\vdash    \hat{a}(\fat{\alpha},1) = \hat{a}(1,\fat{\alpha})$}
\UnaryInfC{$\vdots$}
\RightLabel{$cut$}
\BinaryInfC{$\vdots$}
\AxiomC{$\varphi(\fat{\alpha},1,n)$}
\dottedLine
\UnaryInfC{$\vdash \hat{a}(\fat{\alpha}, \hat{a}(1,n)) = \hat{a}(\hat{a}(\fat{\alpha},1),n)$}
\RightLabel{$cut$}
\BinaryInfC{$\vdash \hat{a}(\fat{\alpha}, \hat{a}(1,n)) = \hat{a}(n',\fat{\alpha})$}
\RightLabel{$\mathcal{E}$}
\UnaryInfC{$\vdash \hat{a}(\fat{\alpha}, s(\hat{a}(0,n))) = \hat{a}(n',\fat{\alpha})$}
\RightLabel{$\mathcal{E}$}
\UnaryInfC{$\vdash \hat{a}(\fat{\alpha}, n') = \hat{a}(n',\fat{\alpha})$}
\end{prooftree}

Notice that $\xi$ is the least element of the order $\prec$ and the following relations concerning $\prec$ are  also defined: $\xi\prec\varphi$, $\xi\prec\psi$, $\psi\prec\chi$. Also, once again $\Phi'$ defines a
\mvLKS-derivation rather than a proof, however, we can perform a similar extension as before to construct a proof. Furthermore we can quantify the passive parameters of the schema and derive the precise statement of commutativity as one would derive in \PA. 
\end{example}

The derivation outlined in Example~\ref{ex:two} differs from the derivations presented in previous work in two major ways: the base cases of the components are allowed to link to other proofs in the schema, and links to components which do not contain the free parameter is sensible. What we mean by the latter remark is that a link to a component with out a free parameter, using the formalism of Section~\ref{sec:basic}, could easily be removed given that it is a link to a non-schematic proof. The active parameter free links in Example~\ref{ex:two} cannot be removed because doing so introduces a second active parameter thus violating the construction. This introduces the difference between active and passive. While the free parameters of Section~\ref{sec:basic} allow us to mimic induction recursively, they do not allow one to introduce lemmata as is done in Example~\ref{ex:two} nor do they allow for basecases dependent on the lemmata. The passive active distinction allows us to circumvent this issue by marking the end of an inductive argument. 

We can formalize the construction of an \mvLKS-proof from an \mvLKS-derivation as follows:

\begin{theorem}
Let  $\Phi =\left\langle \mathbf{C}_1, \cdots, \mathbf{C}_\alpha \right\rangle $  be a $\mathbf{P}$-schema resulting in an  \mvLKS-derivation such that $0<|\mathcal{V}_{i}(es(\Phi))|+|\mathcal{V}_{a}(es(\Phi))|$ and $\mathcal{V}_{p}(es(\Phi))|+ |\mathcal{V}_{i}(es(\Phi))|+|\mathcal{V}_{a}(es(\Phi))| \leq |P|$. Then there exists  $\Phi' =\left\langle \mathbf{C} , \mathbf{C}_1, \cdots, \mathbf{C}_\alpha \right\rangle $ resulting in an \mvLKS-proof.
\end{theorem}
\begin{proof}
Essentially the translation of Example~\ref{example.1}.
\end{proof}
While it might seem unnecessary to restrict ourselves to proofs with end sequents which only contain passive parameters, the restriction provides a simple definition of evaluation presented in the next section. We will refer to such proof schemata as {\em complete} if every sub $\mathbf{P}$-schema of $\Phi$ is also non-computational. 

In the following sections we show that the results and concepts of~\cite{ALeitsch2017} can be extended to our more general formalization.

\subsection{Evaluating $\mathbf{P}$-Schemata}
\label{sec:evalInter}
Like the proof schema of previous work, $\mathbf{P}$-schema represents infinite sequences of proofs. We extend the  soundness result of proof schema  to the  $\mathbf{P}$-schema case. However, the evaluation procedure is a bit more involved given the more complex schematic language of $\mathbf{P}$-schemata.

\begin{definition}[Evaluation of $\mathbf{P}$-schemata] Let $\Phi $ be a complete $\mathbf{P}$-schema $ \left\langle \mathbf{C}_1, \cdots,\right. $ $\left.  \mathbf{C}_\alpha \right\rangle$. We define the rewrite rules for links
\[ \hat{\psi_i}(0, \mathcal{I},\overline{r}) \to \pi \qquad \hat{\psi_i}(s(n), \mathcal{I},\overline{r}) \to \nu \]
where $\hat{\psi_i}$ is the pair of rewrite rules for $\mathbf{C}_i$. The rewrite system for  links is the union of these rules. Moreover, for a substitution $\sigma:\mathcal{N}_p \rightarrow \mathcal{G}_{\mathbb{N}}$ with domain $P$, we define $\Phi\sigma = \hat{\psi_1}\sigma $ as the normal form of $\Phi$ under these rewrite rules and $\mathcal{E}$.
\end{definition}

\begin{lemma} 
\label{lem:unrolling}
Let $\Phi$ be a complete $\mathbf{P}$-schema and $\sigma:\mathcal{N}_p \rightarrow \mathcal{G}_{\mathbb{N}}$ with domain $P$ a substitution.  The rewrite system for the links of $\Phi$ is strongly normalizing and confluent, s.t. $\Phi \sigma $  is an \LK-proof.
\end{lemma}
\begin{proof}
By the restriction on occurrences of  links, a proof schema can be seen as a set of primitive
recursive definitions, and the rewrite rules for links are the standard rules for these definitions. It is well-known that such rewrite systems are strongly normalizing, see \cite{DBLP:journals/entcs/KetemaKO05}.
Finally, by the restriction to complete $\mathbf{P}$-schema links will not occur in the normal form and $\Phi \sigma$ is an \mvLKE-proof. Furthermore, since all $\mathcal{E}$-inferences in this proof are trivial and there are no parameters, we may consider it as an \LK-proof.
\end{proof}
The constraints of Lemma~\ref{lem:unrolling} can be relaxed by allowing computational sub $P$ schema. This  can be handled by structural induction on the $\mathbf{P}$-schema construction by showing that any complete $\mathbf{P}$-schema  could be unrolled into an \LK-proof and therefore is soundly constructed. Also, showing that a $\mathbf{P}$-schema containing computation sub $\mathbf{P}$-schemata, treating any link to a  sub $\mathbf{P}$-schema as an axiom, unrolls into an \LK-proof (modulo a theory extension) and is therefore soundly constructed. 

In light of these details the soundness of $\mathbf{P}$-schemata w.r.t. \mvLKE\ can be stated as follows:
\begin{theorem}[Soundness of $\mathbf{P}$-schemata]
\label{thm:sound}  Let $\Psi$ be a  $\mathbf{P}$-schema  and let $\sigma:\mathcal{N}_p \rightarrow \mathcal{G}_{\mathbb{N}}$ be a substitution. Then $\Psi \sigma $ is an \LK-proof over a theory $T$ of $es(\Psi)\sigma$, where $T$ is the set of end-sequents associated with the computational sub $\mathbf{P}$-schema $\Psi$. 
\end{theorem}
 An interesting side note is that the calculus presented in \cite{DCerna2017a} can easily integrate the formalization presented here and thus can provide a calculus for $\mathbf{P}$-schema. We will consider this in future work. 

Note that when measuring the size of  $\mathbf{P}$-schema, computational sub $\mathbf{P}$-schema of the $\mathbf{P}$-schema do not contribute to the measurement. This will be important for the definition of the {\em numeric} $\mathbf{P}$-schema introduced in Section \ref{sec:krei}.
\section{Local Induction and \mvLKE }
\label{LocalINC}
In~\cite{ALeitsch2017}, it was shown that proof schemata are equivalent to a particular fragment of  arithmetic, i.e.\ the so called {\em $k$-simple induction}, which limits the number of inductive eigenvariables\footnote{Inductive eigenvariables are eigenvariables occurring in the context of an induction inference rule.} to one. The induction rule is as follows

\begin{prooftree}
\AxiomC{$F(k),\Gamma \vdash \Delta, F(s(k))$}
\RightLabel{$\mathbf{IND}$}
\UnaryInfC{$F(0),\Gamma \vdash \Delta, F(t)$}
\end{prooftree}
where $t$ is a term of the numeric sort such that  $t$ either contains $k$ ($k$ is a free parameter in the sense of~\cite{ALeitsch2017}) or is ground. Adding the above rule to $\LKE$  resulted in the $\mathbf{LKIE}$-calculus. We will refer to this calculus as the $\mathit{simple}\ \mathbf{LKIE}$-calculus. The $\mvLKE$-calculus and $\mathbf{P}$-schema are related to a much more expressive induction rule. Essentially any term $t$ can replace the active parameter $n$ of the auxiliary sequent (including a term containing $n$) and the internal parameters can be instantiated with arbitrary terms. The instantiations must obey the restriction of at most one active parameter per schematic sequent. We refer to the calculus with the following induction rule as the \mvLKIE-calculus: 

\begin{small}
\begin{prooftree}
\AxiomC{$F(n,\mathbf{m}_1,\cdots,\mathbf{m}_{\alpha}),\Gamma \vdash \Delta, F(n',\mathbf{m}_1,\cdots,\mathbf{m}_{\alpha})$}
\RightLabel{$\mathit{mv}\mathbf{IND}$}
\UnaryInfC{$F(0,\mathbf{a}_1,\cdots,\mathbf{a}_{\alpha}),\Gamma \vdash \Delta, F(t,\mathbf{a}_1,\cdots,\mathbf{a}_{\alpha})$}
\end{prooftree}
\end{small}
where the $\mathbf{m}_i$ are internal parameters which can be replaced by any schematic term, that is active, passive, ground, or another internal parameter. 

We consider an \mvLKIE-derivation $\psi$ as an \mvLKIE-proof if the end-sequent of $\psi$ only contains passive parameters. We will first consider strict \mvLKIE-proofs which, like strict links, require preservation of the passive parameters, i.e. all passive parameters used in the proof must show up in the end sequent. 
\begin{example}
Here we present a strict \mvLKIE-proof of the $\lbrace \fat{\alpha}, \fat{\beta} \rbrace$-schema of Example~\ref{ex:two}. Notice how the links are replaced by the induction rules in a similar fashion to the $k$-induction conversion introduced in~\cite{ALeitsch2017}. 
\begin{tiny}
\begin{prooftree}
\AxiomC{$(\nu_1)$}
\AxiomC{$\begin{array}{c}\vdots\\\hat{a}(n, \hat{a}(\mathbf{m},\mathbf{k})) = \hat{a}(\hat{a}(n,\mathbf{m}),\mathbf{k})\vdash \\ \hat{a}(n', \hat{a}(\mathbf{m},\mathbf{k})) = \hat{a}(\hat{a}(n',\mathbf{m}),\mathbf{k})  \end{array}$}
\RightLabel{$\mathit{mv}\mathbf{IND}$}
\UnaryInfC{$\begin{array}{c}\hat{a}(0, \hat{a}(1,n)) = \hat{a}(\hat{a}(0,1),n)\vdash \\ \hat{a}(\fat{\alpha}, \hat{a}(1,n)) = \hat{a}(\hat{a}(\fat{\alpha},1),n)  \end{array}$}

\UnaryInfC{$\begin{array}{c}\vdots\\ \vdash \hat{a}(\fat{\alpha}, \hat{a}(1,n)) = \hat{a}(\hat{a}(\fat{\alpha},1),n) \end{array}$}
\RightLabel{$cut$}
\BinaryInfC{$\begin{array}{c}  \hat{a}(\fat{\alpha},n) =\hat{a}(n,\fat{\alpha}) \vdash\\ \hat{a}(\fat{\alpha}, \hat{a}(1,n)) = \hat{a}(n',\fat{\alpha}) \end{array}$}
\RightLabel{$\mathcal{E}$}
\UnaryInfC{$\begin{array}{c} \hat{a}(\fat{\alpha},n) =\hat{a}(n,\fat{\alpha}) \vdash\\ \hat{a}(\fat{\alpha}, s(\hat{a}(0,n))) = \hat{a}(n',\fat{\alpha}) \end{array}$}
\RightLabel{$\mathcal{E}$}
\UnaryInfC{$\begin{array}{c} \hat{a}(\fat{\alpha},n) =\hat{a}(n,\fat{\alpha}) \vdash \\ \hat{a}(\fat{\alpha}, n') = \hat{a}(n',\fat{\alpha}) \end{array}$}
\RightLabel{$\mathit{mv}\mathbf{IND}$}
\UnaryInfC{$\begin{array}{c} \hat{a}(\fat{\alpha},0) =\hat{a}(0,\fat{\alpha}) \vdash\\ \hat{a}(\fat{\alpha}, \fat{\beta} ) = \hat{a}(\fat{\beta},\fat{\alpha}) \end{array}$}
\end{prooftree}
\end{tiny}
where $\nu_1$ is
\begin{tiny}
\begin{prooftree}
\AxiomC{$\begin{array}{c}  \hat{a}(\fat{\alpha},n) =\hat{a}(n,\fat{\alpha}) \vdash \\  \hat{a}(\fat{\alpha},n) =\hat{a}(n,\fat{\alpha}) \end{array} $}
\AxiomC{$\begin{array}{c} \hat{a}(n,1) = \hat{a}(1,n) \vdash\\    \hat{a}(n',1) = \hat{a}(1,n')  \end{array}$}
\dottedLine
\RightLabel{$\mathit{mv}\mathbf{IND}$}
\UnaryInfC{$\begin{array}{c}  \hat{a}(0,1) = \hat{a}(1,0) \vdash \\   \hat{a}(\fat{\alpha},1) = \hat{a}(1,\fat{\alpha})  \end{array}$}
\UnaryInfC{$\vdots$}
\RightLabel{$cut$}
\BinaryInfC{$\vdots$}
\end{prooftree}
\end{tiny}
\end{example}

Note that the proof of equivalence between  $\mathbf{LKIE}$-proofs and proof schemata provided in~\cite{ALeitsch2017} does not directly use the  $k$-simple induction restriction. What is important for the argumentation is that $\mathbf{LKIE}$-proofs are structured in similar fashion as proof schemata.  Thus we can, for the most part, use the same arguments to prove the feasibility of translation for $\mathbf{P}$-schema and \textit{mv} induction. This argument is easier to make when we enforce the proofs to be strict, i.e. transforming them into complete $\mathbf{P}$-schema, however such construction can be easily generalized to complete $\mathbf{P}$-schema by considering  strict $\mathbf{P}$-schema as a base case for well-founded structural induction. 

\begin{lemma}
\label{lem:SI}
Let $\Psi$ be a strict\footnote{$\mathcal{V}_{p}(es(\Psi)) )\equiv P$} $\mathbf{P}$-schema with end-sequent $\mathcal{S}$. Then there exists a strict \mvLKIE-derivation of $\mathcal{S}$.
\end{lemma}
\begin{proof}
 We can start by considering the proof of Proposition 3.13 from~\cite{ALeitsch2017}. We know that proof schemata are  $\lbrace\fat{\alpha}\rbrace$-schemata and only contain $(n, \mathcal{I} )$-components. Thus, Proposition 3.13 from~\cite{ALeitsch2017} provides a base case for the translation of strict $\mathbf{P}$-schema to \mvLKIE-derivations. Now let us consider the case when we have more than one active parameter. This does not influence the construction outlined in~\cite{ALeitsch2017} because there is only one active parameter per component. The only difference is that the translation of~\cite{ALeitsch2017} must be applied to the base cases of the components as well. 
\end{proof}
Note that the translation defined in Lemma~\ref{lem:SI}, along with Lemma~\ref{lem:one}, almost provides equivalence between strict $\mathbf{P}$-schema and primitive recursive arithmetic (\PRA), though  it is not clear if the $\mathcal{E}$ rule provides a more expressive language than \PRA.  However, Lemma~\ref{lem:one} clearly shows that $\PRA \subseteq \ \mathrm{strict}\ \mathbf{P}$-schema. 
\begin{lemma}
\label{lem:one}
Let $\Pi$ be a strict \mvLKIE-derivation of $\mathcal{S}$ containing $\alpha$ induction inferences of the form   
\begin{small}
\begin{prooftree}
\AxiomC{$F_{\beta}(n,\mathbf{m}_1,\cdots,\mathbf{m}_{\gamma}),\Gamma_{\beta} \vdash \Delta_{\beta}, F_{\beta}(n',\mathbf{m}_1,\cdots,\mathbf{m}_{\gamma})$}
\RightLabel{$\mathit{mv}\mathbf{IND}$}
\UnaryInfC{$F_{\beta}(0,\mathbf{a}_1,\cdots,\mathbf{a}_{\gamma}),\Gamma_{\beta} \vdash \Delta_{\beta}, F_{\beta}(t,\mathbf{a}_1,\cdots,\mathbf{a}_{\gamma})$}
\end{prooftree}
\end{small}
where $1 \leq \beta \leq \alpha$, and if $\eta < \beta$ then the induction inference 
with conclusion 
$$F_{\beta}(0,\mathbf{a}_1,\cdots,\mathbf{a}_{\gamma}),\Gamma_{\beta} \vdash \Delta_{\beta}, F_{\beta}(t,\mathbf{a}_1,\cdots,\mathbf{a}_{\gamma})$$
is above the induction inference with conclusion 
$$F_{\eta}(0,\mathbf{a}_1',\cdots,\mathbf{a}_{\gamma*}'),\Gamma_{\eta} \vdash \Delta_{\eta}, F_{\eta}(t,\mathbf{a}_1',\cdots,\mathbf{a}_{\gamma*}')$$ in $\Pi$.  Then there exists a strict $\mathbf{P}$-schema with end-sequent $\mathcal{S}$.
\end{lemma}
\begin{proof}
Let $T$ be the transformation taking an \mvLKIE-derivation $\varphi$ to an \mvLKS-derivation by 
replacing the induction inferences $F_{\eta}(0,\mathbf{a}_1,$ $\cdots,$ $\mathbf{a}_{\gamma}),$ $
\Gamma_{\eta}$ $\vdash$ $\Delta_{\eta},$ $ F_{\eta}(t,\mathbf{a}_1,$ $\cdots,$ $\mathbf{a}_{\gamma}')$, $\eta < \beta$,  with a proof link $\psi_\eta(t,\mathbf{a}_1',\cdots,
\mathbf{a}_{\gamma'}')$.  If the transformation reaches the induction inference $\beta$ it 
replaces the $F_{\beta}(0,\mathbf{a}_1,$ $\cdots,$ $\mathbf{a}_{\gamma}),$ $\Gamma_{\beta}$ $
\vdash$ $\Delta_{\beta},$ $F_{\beta}(t,$ $\mathbf{a}_1,$ $\cdots,$ $\mathbf{a}_{\gamma})$ with a 
proof link $\psi_\beta(t,\mathbf{m}_1,\cdots,$ $\mathbf{m}_{\gamma})$ and sequent of the 
proof link $F_{\beta}(0,\mathbf{m}_1,$ $\cdots,$ $\mathbf{m}_{\gamma}),$ $\Gamma_{\beta}$ $\vdash$ 
$\Delta_{\beta}, F_{\beta}(n,\mathbf{m}_1,$ $\cdots,$ $\mathbf{m}_{\gamma}).$ The instantiation is 
placed in the construction of the component for the predecessor of $\beta$. Such a transformation obviously constructs a  strict $\mathbf{P}$-schema from a strict \mvLKIE-derivation. 

We will inductively construct a  strict $\mathbf{P}$-schema $ \left\langle \mathbf{C}_1, \cdots , \right. $ $\left. \mathbf{C}_\alpha\right\rangle$  where $\mathbf{C}_\beta = (\psi_\beta, \pi,\nu)$ has the end sequent $F_{\beta}(0,\mathbf{m}_1,\cdots$ $,\mathbf{m}_{\gamma}),\Gamma_{\beta} $ $\vdash \Delta_{\beta}, F_{\beta}(n,\mathbf{m}_1,\cdots,\mathbf{m}_{\gamma})$ for some active parameter $n$. Assume that we have
already constructed such proofs for $\mathbf{C}_{\beta+1}, \cdots , \mathbf{C}_\alpha$ and consider the induction inference with the following main sequent $F_{\beta}(0,\mathbf{a}_1,\cdots,\mathbf{a}_{\gamma}),\Gamma_{\beta} $ $ \vdash \Delta_{\beta}, F_{\beta}(t,\mathbf{a}_1,$ $\cdots,\mathbf{a}_{\gamma})$. Let $\xi$ be the derivation above the induction. We set $\pi$ to  $F_{\beta}(0,\mathbf{m}_1,$ $\cdots,$ $\mathbf{m}_{\gamma}),$ $\Gamma_{\beta}$ $\vdash$ $\Delta_{\beta},$ $ F_{\beta}(0,$ $\mathbf{m}_1,$ $\cdots,$ $\mathbf{m}_{\gamma})$  which by
definition fulfills the requirements of links. Furthermore, let $\nu$ be the proof
\begin{small}
\begin{prooftree}
\AxiomC{$\begin{array}{c}\psi_{\beta}(n,\bar{\mathbf{m}}_{\gamma})\\ S(\psi_{\beta}(n,\bar{\mathbf{m}}_{\gamma}))\end{array}$}
\AxiomC{$\begin{array}{c}T(\xi)\\S(T(\xi))\end{array}$}
\RightLabel{$\mathit{mv}\mathbf{IND}$}
\BinaryInfC{$F_{\beta}(0,\bar{\mathbf{m}}_{\gamma}),\Gamma_{\beta},\Gamma_{\beta} \vdash \Delta_{\beta}, \Delta_{\beta}, F_{\beta}(n',\bar{\mathbf{m}}_{\gamma})$}
\RightLabel{$c*$}
\UnaryInfC{$F_{\beta}(0,\bar{\mathbf{m}}_{\gamma}),\Gamma_{\beta} \vdash  \Delta_{\beta}, F_{\beta}(n',\bar{\mathbf{m}}_{\gamma})$}
\end{prooftree}
\end{small}
where $S(\psi_{\beta}(n,\bar{\mathbf{m}}_{\gamma})) \equiv F_{\beta}(0,\bar{\mathbf{m}}_{\gamma}),\Gamma_{\beta} \vdash \Delta_{\beta}, F_{\beta}(n,\bar{\mathbf{m}}_{\gamma})$, $S(T(\xi$ $)) \equiv F_{\beta}(n,\bar{\mathbf{m}}_{\gamma}),\Gamma_{\beta} \vdash \Delta_{\beta}, F_{\beta}(n',\bar{\mathbf{m}}_{\gamma})$, which also clearly satisfies the requirement on  links. Summarizing, $\mathbf{C}_{\beta}$ is a component with end-sequent $F_{\beta}(0,\bar{\mathbf{m}}_{\gamma}),\Gamma_{\beta} \vdash \Delta_{\beta}, F_{\beta}(n,\bar{\mathbf{m}}_{\gamma})$. Linkability and the partial ordering come for free from the construction of strict \mvLKIE-derivations. 
\end{proof}

Notice that  Lemma~\ref{lem:one}  does not put a restriction on the number of passive parameters in the end  sequent, but limits the partial ordering of components to a total linear ordering. A simple corollary of  Lemma~\ref{lem:one} removes the restriction on the ordering of components. Notice that proving the corollary requires the same induction argument over a more complex order structure (the linkability ordering). Essentially, we would have to join chains of components together using cuts.

\begin{corollary}
\label{thm:one}
Let $\Pi$ be a strict \mvLKIE-derivation of $\mathcal{S}$. Then there exists a strict $\mathbf{P}$-schema with end-sequent $\mathcal{S}$.
\end{corollary}

Concerning strict \mvLKIE-derivations, notice that the need for passive, internal, and active parameters is no longer there. The three parameters aided the formalization of $\mathbf{P}$-schema by removing mutual recursion and parameter instantiation, which are difficult to handle. Essentially, a reasonable class of  strict $\mathbf{P}$-schema could not be constructed without the three types of parameters. But for strict \mvLKIE-derivations, the construction is obvious and enforced by the proof structure, thus, we can replace internal and active parameters by the corresponding constants and passive parameters. The resulting rule is essentially the induction rule of arithmetic. However, along with the $\mathcal{E}$ rule, the language is at least a conservative extension of \PRA. We show that for a particular choice of equational theory and using the standard equational axioms, the  $\mathcal{E}$ rule is admissible and thus strict \mvLKIE-derivations are at least as expressive as {\PRA} and by transitivity so is the strict $\mathbf{P}$-schema formulation. 

Furthermore, by dropping the strictness requirement, that is allowing computational sub $\mathbf{P}$-schema, we can show that the $\mathbf{P}$-schema formulation is at least as expressive as \PA. As we mention above, this can be done by using the results of this section as a base case and performing a structural induction over the construction of a $\mathbf{P}$-schema.

\section{ $\mathbf{P}$-schema, \PRA , and \PA}
\label{PAandNS}
We will consider the $\mathbf{P}$-schema formulation over the following equational theory
\[\mathcal{E}_{\rm PA} = \left\lbrace \begin{array}{c} \widehat{a}(s(n),\beta) =s(\widehat{a}(n,\beta)) ; \ \widehat{a}(0,\beta) = \beta  \\
\widehat{m}(s(n),\beta) =\widehat{a}(\widehat{m}(n,\beta),\beta) ; \ \widehat{m}(0,\beta) = 0
\end{array} \right\rbrace \]
using the axioms of the formalization of Peano arithmetic and equational theory found in~\cite{GTakeuti1975}. We will refer to these axioms as \AxPa. We only consider a single two place propositional symbol which we will refer to as equality. We will refer to the calculus defined in~\cite{GTakeuti1975} for {\PA} as the \PA-calculus.

\begin{lemma}
\label{lem:Efree}
Let $\Pi$ be a strict \mvLKIE-proof using \AxPa\ and \EPa. Then there exists a strict \mvLKIE-proof\ \ $\Pi'$ without the $\mathcal{E}$ inference rule ($\Pi'$ is $\mathcal{E}$-free). 
\end{lemma}
\begin{proof}
The rewrite rules of \EPa\ are precisely the axioms of \AxPa\ for addition and multiplication. Thus, from those axioms and the equational theory found in~\cite{GTakeuti1975} anything provable by the $\mathcal{E}$ inference rule can be proven using the above mentioned axioms and atomic cuts. 
\end{proof}
Now that we have $\mathcal{E}$-free strict \mvLKIE-proofs we can consider translation to the \PA-calculus without quantification. As the end sequent of a strict \mvLKIE-proof only has passive variables. We can push the passive variables up the proof tree and replace each active variable by a fresh passive variable. Thus, the resulting proof only contains passive variables and constants and is a proof in the \PRA-calculus, being that we have so far avoided quantification. Of course, a back translation can be performed, however doing so would not result in the same proof as the one we started with. 
\begin{theorem}
\label{thm:PRAEQ}
There exists a $\mathcal{E}$-free strict \mvLKIE-proof of a sequent $S$ iff there exists a  \PRA-calculus proof of $S$. 
\end{theorem}
\begin{proof}
As a consequence of pushing the passive parameters up the proof to the leaves we convert \mvLKIE
~induction rules to standard induction rules, and thus \PRA-calculus proofs. Furthermore a back translation is possible by reversing the method.
\end{proof}
\begin{example}
The following proof of commutativity is the result of applying the translation from $\mathcal{E}$-free strict \mvLKIE-proofs to \PA-calculus proofs to the example first illustrated in Example~\ref{ex:two}. 
\begin{tiny}
\begin{prooftree}
\AxiomC{$(\nu_1)$}
\AxiomC{$\begin{array}{c}\vdots\\\hat{a}(\fat{\nu}, \hat{a}(1,\fat{\gamma})) = \hat{a}(\hat{a}(\fat{\nu},1),\fat{\gamma})\vdash \\ \hat{a}(s(\fat{\nu}), \hat{a}(1,\fat{\gamma})) = \hat{a}(\hat{a}(s(\fat{\nu}),1),\fat{\gamma})  \end{array}$}
\RightLabel{$\mathbf{IND}$}
\UnaryInfC{$\begin{array}{c}\hat{a}(0, \hat{a}(1,\fat{\gamma})) = \hat{a}(\hat{a}(0,1),\fat{\gamma})\vdash \\ \hat{a}(\fat{\alpha}, \hat{a}(1,\fat{\gamma})) = \hat{a}(\hat{a}(\fat{\alpha},1),\fat{\gamma})  \end{array}$}

\UnaryInfC{$\begin{array}{c}\vdots\\ \vdash \hat{a}(\fat{\alpha}, \hat{a}(1,\fat{\gamma})) = \hat{a}(\hat{a}(\fat{\alpha},1),\fat{\gamma}) \end{array}$}
\RightLabel{$cut$}
\BinaryInfC{$\begin{array}{c}  \hat{a}(\fat{\alpha},\fat{\gamma}) =\hat{a}(\fat{\gamma},\fat{\alpha}) \vdash\\ \hat{a}(\fat{\alpha}, \hat{a}(1,\fat{\gamma})) = \hat{a}(s(\fat{\gamma}),\fat{\alpha}) \\ (1)\end{array}$}

\end{prooftree}
\end{tiny}
where $\nu_1$ is
\begin{tiny}
\begin{prooftree}
\AxiomC{$\begin{array}{c}  \hat{a}(\fat{\alpha},\fat{\gamma}) =\hat{a}(\fat{\gamma},\fat{\alpha}) \vdash \\  \hat{a}(\fat{\alpha},\fat{\gamma}) =\hat{a}(\fat{\gamma},\fat{\alpha}) \end{array} $}
\AxiomC{$\begin{array}{c} \hat{a}(\fat{\mu},1) = \hat{a}(1,\fat{\mu}) \vdash\\    \hat{a}(s(\fat{\mu}),1) = \hat{a}(1,s(\fat{\mu}))  \end{array}$}
\dottedLine
\RightLabel{$\mathbf{IND}$}
\UnaryInfC{$\begin{array}{c}  \hat{a}(0,1) = \hat{a}(1,0) \vdash \\   \hat{a}(\fat{\alpha},1) = \hat{a}(1,\fat{\alpha})  \end{array}$}
\UnaryInfC{$\vdots$}
\RightLabel{$cut$}
\BinaryInfC{$\vdots$}
\end{prooftree}
\end{tiny}

\begin{tiny}
\begin{prooftree}
\AxiomC{$(\nu_2)$}
\AxiomC{$\begin{array}{c}\vdots\\  \hat{a}(\fat{\alpha}, s(\hat{a}(0,\fat{\gamma}))) = \hat{a}(s(\fat{\gamma}),\fat{\alpha}),\\ \fat{\alpha} = \fat{\alpha}, s(\fat{\gamma}) = s(\fat{\gamma}), \\ s(\hat{a}(0,\fat{\gamma})) = s(\fat{\gamma}) \vdash\\ \hat{a}(\fat{\alpha}, s(\fat{\gamma})) = \hat{a}(s(\fat{\gamma}),\fat{\alpha}) \end{array}$}
\RightLabel{$cut$}
\BinaryInfC{$\vdots$}
\UnaryInfC{$\begin{array}{c} \hat{a}(\fat{\alpha},\fat{\gamma}) =\hat{a}(\fat{\gamma},\fat{\alpha}) \vdash \\ \hat{a}(\fat{\alpha}, s(\fat{\gamma})) = \hat{a}(s(\fat{\gamma}),\fat{\alpha}) \end{array}$}
\RightLabel{$\mathbf{IND}$}
\UnaryInfC{$\begin{array}{c} \hat{a}(\fat{\alpha},0) =\hat{a}(0,\fat{\alpha}) \vdash\\ \hat{a}(\fat{\alpha}, \fat{\beta} ) = \hat{a}(\fat{\beta},\fat{\alpha}) \end{array}$}
\end{prooftree}
\end{tiny}
where $\nu_2$ is
\begin{tiny}
\begin{prooftree}
\AxiomC{$\begin{array}{c}(1)\\  \hat{a}(\fat{\alpha},\fat{\gamma}) =\hat{a}(\fat{\gamma},\fat{\alpha}) \vdash\\ \hat{a}(\fat{\alpha}, \hat{a}(1,\fat{\gamma})) = \hat{a}(s(\fat{\gamma}),\fat{\alpha}) \end{array}$}
\AxiomC{$\begin{array}{c}  \hat{a}(\fat{\alpha}, \hat{a}(1,\fat{\gamma})) = \hat{a}(s(\fat{\gamma}),\fat{\alpha}),\\ \fat{\alpha} = \fat{\alpha}, s(\fat{\gamma}) = s(\fat{\gamma}), \\ \hat{a}(1,\fat{\gamma}) = s(\hat{a}(0,\fat{\gamma})) \vdash\\ \hat{a}(\fat{\alpha}, s(\hat{a}(0,\fat{\gamma}))) = \hat{a}(s(\fat{\gamma}),\fat{\alpha}) \end{array}$}
\RightLabel{$cut$}
\BinaryInfC{$\vdots$}
\UnaryInfC{$\begin{array}{c} \hat{a}(\fat{\alpha},\fat{\gamma}) =\hat{a}(\fat{\gamma},\fat{\alpha}) \vdash\\ \hat{a}(\fat{\alpha}, s(\hat{a}(0,\fat{\gamma}))) = \hat{a}(s(\fat{\gamma}),\fat{\alpha}) \end{array}$}
\end{prooftree}
\end{tiny}
\end{example}

To get from {\PRA} to {\PA} we need to remove the requirement of only considering  strict \mvLKIE-proofs. In terms of $P$-schema, this would mean allowing computational sub $P$-schema. We can build a \mvLKIE-proof $\chi$ containing a sub-derivation $\psi$ which is a strict \mvLKIE-proofs by allowing strong quantification on the passive parameters of $\psi$ in $\chi$. The problem is that doing so can possibly destroy the translation of Section~\ref{LocalINC}. To show this is not possible we just have to consider translation of $\chi$ in parts, first we translate $\psi$ and then we translate $\chi$ without $\psi$, that is replacing $\psi$ with a theory axiom during translation. Once we finish the translation of both parts we glue them back together to get a translation of the original proof $\chi$. This argument can be formalized as mentioned in previous sections by performing a structural induction based on the above argument. This results in the  following theorems: 
\begin{theorem}
\label{lem:PASI}
There is a $\mathbf{P}$-schema $\Psi$  with end-sequent $\mathcal{S}$ iff there is a \mvLKIE-derivation of $\mathcal{S}$.
\end{theorem}
\begin{proof}
We can convert $\Psi$ into an \mvLKIE-derivation by structural induction over the number of passive parameters not associated with computational sub $\mathbf{P}$-schema.  First we consider strict $\mathbf{P}$-schema (Theorem~\ref{thm:PRAEQ}). As the IH we assume that the theorem holds for the first $n$ computational sub $\mathbf{P}$-schema of $\Psi$, then we show it for $n+1$.
\end{proof}
Finally, we can extend the results of this section to \PA. 

\begin{theorem}
\label{thm:PAEQ} 
There is an  $\mathcal{E}$-free  \mvLKIE-derivation  of a sequent $S$ iff there is a \PA-calculus proof of $S$. 
\end{theorem}
\begin{proof}
Note that the \PA induction rule is a special case of the $\mvLKIE$ induction rule and thus making backwards translation possible
\end{proof}

\section{Herbrand's Theorem \& Systems: beyond $k$-induction}
\label{sec:krei}
In this section we introduce an extension of the version of Herbrand's theorem presented in~
\cite{CDunchev2014,ALeitsch2017}, the concept of {\em Herbrand systems}, and how to extend this 
concept to $P$-schema. We should note that Herbrand's theorem can only be extended to strict $P$-
schema, that is $P$-schema equivalent to primitive recursive arithmetic. The more general form of 
$P$-schema we introduced earlier allows for computational \textit{sub}$P$-schema   which  encode 
``infinite'' information, that is, in some sense introduce new theory. Constructing Herbrand 
systems for strict $P$-schema requires extending the method of~\cite{CDunchev2014,ALeitsch2017} to 
strict $P$-schema which is currently being investigated. However, if we assume that the  strict $P
$-schemata we are dealing with either only contain propositional cuts, that is quantifier free 
cuts, or are in Inessential \textit{cut} Normal Form (I\textit{cut}NF), a particular type of proof 
with quantifier free cuts, then it is easy to see how previous result can be extended to strict $P
$-schema. The I\textit{cut}NF (See Figure~\ref{normalform}) is based on the concept of {\em characteristic formula schema}~
\cite{ALeitsch2017} (The symbol F$_{i}$ in Figure~\ref{normalform})  which essentially represents the cut structure of a given proof schema. 
Introducing this concept for strict $P$-schema is beyond the scope of this work, and would 
essentially result in a formalism nearly identical to that introduced in~\cite{ALeitsch2017}. 
Therefore we refrain from discussing it here and refer the reader to that work. 

\begin{definition}[Multi-parameter Herbrand system (extension of def. in~\cite{ALeitsch2017})]
\label{HS}
Let $G = \exists x_1, \cdots \exists x_{\alpha} F(n_1,\cdots n_{\beta},$ $ x_1 ,\cdots , x_{\alpha} )$, s.t.
$F(n_1,\cdots n_{\beta}, x_1 ,\cdots , x_{\alpha} )$ quantifier-free and $n_{1},\cdots n_{\beta}: \omega$ are the only free variables of $G$. Then a Herbrand system for $G$ is a rewrite system $R$ (containing the list constructors and a function symbol $W$) such that for all $(\gamma_1,\cdots,\gamma_{\beta}) \in \mathbb{N}$, the normal form of $W(\gamma_1,\cdots,\gamma_{\beta})$ w.r.t. $R$ is a list containing $\epsilon$
lists of terms $t_{\mu,1} , \cdots , t_{\mu,\alpha}$ s.t. $F(\gamma_1,\cdots,\gamma_{\beta},$ $ t_{1,1} , \cdots , t_{1,\alpha}) \vee \cdots \vee  F(\gamma_1,\cdots,\gamma_{\beta},t_{\epsilon,1} , \cdots , t_{\epsilon,\alpha})$ is provable using the \LKE\ calculus.
\end{definition}
The formula $F$ which is being considered in Definition~\ref{HS} is the characteristic formula representing the cut structure of a given proof schema/strict $P$-schema. To actually extend Herbrand's theorem we need to assume the proof has the form mentioned earlier in this section as was done in~\cite{ALeitsch2017}. Transforming a strict $P$-schema ICutNF is left to future work, but it is enough to assume quantifier free cuts.
\begin{figure}
\begin{tiny}
\begin{prooftree}
	\AxiomC{$\phi_\alpha$}
	\UnaryInfC{$\Gamma \vdash \Delta, F_\alpha$}
	\AxiomC{$\phi_2$}
	\UnaryInfC{$\Gamma \vdash \Delta, F_2$}
	\AxiomC{$\phi_1$}
	\UnaryInfC{$\Gamma \vdash \Delta, F_1$}
	\AxiomC{$\Phi$}
	\UnaryInfC{$F_1 , \ldots , F_\alpha \vdash$}
	\RightLabel{$(w : l)$}
	\UnaryInfC{$F_1 , \ldots , F_\alpha, \Gamma \vdash \Delta$}
	\RightLabel{$(cut + c : * )$}
	\BinaryInfC{$\Gamma, F_2 , \ldots , F_\alpha \vdash \Delta$}
	\RightLabel{$(cut + c : * )$}
	\BinaryInfC{$\Gamma, F_3 , \ldots , F_\alpha \vdash \Delta$}
	\UnaryInfC{$\vdots$}
	\RightLabel{$(cut + c : * )$}
	\BinaryInfC{$\Gamma \vdash \Delta$}
\end{prooftree}
\end{tiny}
\caption{Illustration of an I\textit{cut}NF where $\phi_i$ is an instance of the projection schema for a substitution $\sigma_i$, $F_{i}$ is the result of applying $\sigma_i$ to the characteristic formula, and $\Phi$ is a proof of the Herbrand conjunction $F_1 , \ldots , F_\alpha \vdash$. }
\label{normalform}
\end{figure}

\begin{theorem}[Extension of Herbrand's theorem to strict $P$-schema]
Let $$ G = \exists x_1, \cdots \exists x_{\alpha} F(n_1,\cdots n_{\beta}, x_1 ,\cdots , x_{\alpha} ) $$ and assume that $\Psi$  is a strict $P$-schema, containing only quantifier-free cuts, with end-sequent
$\vdash G$. There exists a Herbrand system $R$ for $G$ such that $R$ is of linear size w.r.t. $\Psi$.
\end{theorem}
\begin{proof}
The proof essentially follows the same steps as the proof in~\cite{ALeitsch2017} except we must consider a vector of parameters instead of a single parameter. This will introduce one extra induction over the number of parameters.
\end{proof}

\section{Conclusion}
\label{CON}
In this paper we generalized the proof schemata formalism of~\cite{ALeitsch2017} to a much larger fragment of arithmetic. Our extension does not effect the applicability of the proof analysis tools outlined in~\cite{CDunchev2014}. Also, we show that our new formalism is at least as expressive as Peano arithmetic. 
We conjecture that it is a conservative extension of arithmetic and plan to address this question in future work. 

Furthermore, we plan to investigate extensions of the method of~\cite{ALeitsch2017} to this more general formalism being that it provides a cut-elimination complete method for proof schemata. Also, as was addressed in~\cite{DCerna2017a}, we would like to develop a calculus for construction of $\mathbf{P}$-schema directly and in the process develop compression techniques for components used in various locations in the same $\mathbf{P}$-schema. One topic concerning proof schemata which has not been investigated is using inductive definitions other than the natural numbers to index the proof. We plan to investigate generalizations of the indexing sort in future work.  

Also, we provide a generalization of Herbrand's theorem to strict $P$-schema based on the relate work done in~\cite{ALeitsch2017}. While this extension does not extend to {\PA} it does extend to \PRA. Concerning conservative reflection principles for theories equivalent to {\PA} \cite{RParikh1973}, this work provides an interesting and non-trivial example of such principles  providing an alternative and advantageous perspective of the theory of \PA. So far most research into schematic formalisms has been focused on proof transformation, the area which gave birth the concept. By providing an equivalence result with a strong arithmetic theory, we hope others will find interest in this formalism. 

{\textbf{Acknowledgements:} We would like to thank  Michael Lettmann, Alexander Leitsch and Matthias Baaz for contributing to the polishing of this work.}
\bibliographystyle{plain}
\bibliography{References}

\end{document}